\documentclass[11pt]{amsart}

\usepackage{graphicx}
\usepackage{graphics,epsfig}
\usepackage{psfrag}

\usepackage{color}
\usepackage{latexsym}
\usepackage{amsmath}
\usepackage{amsfonts}
\usepackage{amssymb}
\usepackage{esint}
\usepackage[all]{xy}
\renewcommand{\a }{\alpha }
\renewcommand{\b }{\beta }
\renewcommand{\d}{\delta }
\newcommand{\D }{\Delta }

\newcommand{\e }{\varepsilon }
\newcommand{\g }{\gamma}

\newcommand{\G }{\Gamma }
\renewcommand{\l }{\lambda }
\renewcommand{\L }{\Lambda }

\newcommand{\n }{\nabla }

\newcommand{\s }{\sigma }

\renewcommand{\t }{\tau }

\renewcommand{\O }{\Omega }

\newcommand{\ov}{\overline}

\def\p{\partial}

\newcommand{\wtilde }{\widetilde}

\newcommand{\be}{\begin{equation}}
\newcommand{\ee}{\end{equation}}

\newcommand{\R}{\mathbb{R}}
\renewcommand{\S}{\mathbb{S}}

\newcommand{\N}{\mathbb{N}}

\newcommand{\ba}{\boldsymbol{\a}}

\newtheorem{theorem}{Theorem}[section]
\newtheorem{proposition}[theorem]{Proposition}
\newtheorem{definition}{Definition}[section]

\newtheorem{example}[theorem]{Example}

\newcommand{\bpr}{\begin{proposition}}
\newcommand{\epr}{\end{proposition}}
\newcommand{\bex}{\begin{example}\rm}
\newcommand{\eex}{\end{example}}

\everymath{\displaystyle}

\begin{document}

\newtheorem{lem}{Lemma}[section]
\newtheorem{pro}[lem]{Proposition}
\newtheorem{thm}[lem]{Theorem}
\newtheorem{rem}[lem]{Remark}
\newtheorem{cor}[lem]{Corollary}
\newtheorem{df}[lem]{Definition}

\title[Curvature with conical singularities and geodesic boundary]{Prescribing Gaussian curvature on surfaces with conical singularities and \\geodesic boundary}

\author[L. Battaglia]{Luca Battaglia}
\address{Luca Battaglia, Dipartimento di Matematica e Fisica, Universit\`a degli Studi Roma Tre, Largo S.
Leonardo Murialdo, 00146 Roma, Italy}
\email{lbattaglia@mat.uniroma3.it}

\author[A. Jevnikar]{ Aleks Jevnikar}
\address{Aleks Jevnikar, Department of Mathematics, Computer Science and Physics, University of Udine, Via delle Scienze 206, 33100 Udine, Italy}
\email{aleks.jevnikar@uniud.it}

\author[Z.A. Wang]{Zhi-An Wang}
\address{Zhi-An Wang, Department of Applied Mathematics, Hong Kong Polytechnic University,
Hung Hom, Kowloon, Hong Kong}
\email{mawza@polyu.edu.hk}

\author[W. Yang]{ Wen Yang}
\address{ Wen ~Yang,~Wuhan Institute of Physics and Mathematics, Innovation Academy for Precision Measurement Science and Technology, Chinese Academy of Sciences, Wuhan 430071, P. R. China.}
\email{math.yangwen@gmail.com}


\thanks{W. Yang is partially supported by NSFC No. 12171456, No. 12271369 and No. 11871470.}

\begin{abstract}
We study conformal metrics with prescribed Gaussian curvature on surfaces with conical singularities and geodesic boundary in supercritical regimes. Exploiting a variational argument, we derive a general existence result for surfaces with at least two boundary components. This seems to be the first result in this setting. Moreover, we allow to have conical singularities with both positive and negative orders, that is cone angles both less and greater than $2\pi$.
\end{abstract}
\maketitle

\section{Introduction} \label{sec:intro}

\medskip

The prescribed Gaussian curvature problem on a compact surface $M$ under a conformal change of the metric is a classical problem in geometry dating back to Berger \cite{ber}, Kazdan-Warner \cite{kaz-war} and Chang-Yang \cite{chang-yang1, chang-yang2}. Its singular analog on surfaces with conical singularities has been already considered by Picard \cite{pic} and it was later systematically studied by Troyanov \cite{troy}. This problem has been studied for
several decades and there is by now a huge literature on it, see for example \cite{chen-li1,chen-li2,chen-li3,luo-tian,mcow} and the more recent results by Malchiodi and his collaborators using PDE methods \cite{bdmm,bm,bt,car-mal,mal-ru} or by Eremenko, Mondello and Panov using a geometric argument \cite{er,mond-pan1,mond-pan2}. See also Mazzeo-Zhu \cite{mazz-zhu} for a different approach.

\medskip

If $M$ has a boundary, it is then natural to prescribe also the geodesic curvature on $\p M$. For this problem we still do not have a complete picture and there are fewer results mainly concerning the regular case, see \cite{bwz,brend,chang-yang2,che,hw,jim} and the recent results by Malchiodi, Ruiz and their group \cite{bmp,cruz-ru,jlsmr,lsmr}.

The higher dimensional analogue is the well-known problem of prescribing the
scalar curvature on a manifold and mean curvature on the boundary. In particular, the scalar flat case with constant mean curvature takes the name of Escobar problem which has a deep relation with the classical Yamabe problem, see for example \cite{alm,esc,hl,mar} and the references therein.

\medskip

We are interested here in the singular flat geodesic case, namely we study conformal metrics with prescribed Gaussian curvature on surfaces with conical singularities and geodesic boundary. It seems that the only result in this direction is the one by Troyanov \cite{troy} asserting the existence of such metrics in the subcritical case, see the discussion in the sequel. The goal of this paper is to give a first existence result in the supercritical regime, which holds for a large class of singular surfaces. To state it, we need to introduce some notation.

Let $g$ be a metric on $M$. A point $p\in M$ is a conical singularity of order $\a\in(-1,+\infty)$, or angle $\theta_\a=2\pi(1+\a)$, for the metric $g$ if
$$
g(z)=\rho(z)|z|^{2\a}|dz|^2 \quad \mbox{locally around } p,
$$
for some continuous positive function $\rho$. We collect the set of conical singularities $p_j$ of orders $\a_j$ in the formal sum
$$
 \ba=\sum_{j=1}^N \a_j p_j
$$
and denote by $(M,\ba)$ the surface with that set of conical singularities. An important quantity in this study is played by the singular Euler characteristic
$$
\chi(M,\ba)=\chi(M)+\sum_{j=1}^N \a_j.
$$
Here $\chi(M)$ is the Euler characteristic of $M$, that is $\chi(M)=2-\mathfrak{g}-\mathfrak{b}$, where $\mathfrak{g}$ is the genus and $\mathfrak{b}$ is the number of boundary components of $M$. The critical regime of a singular surface is related to the Moser-Trudinger inequality. Following Troyanov \cite{troy} and recalling that we have non-empty boundary, we denote the Trudinger constant of $(M,\ba)$ by
$$
\t(M,\ba)=1+\min_j\left\{\a_j,0\right\}
$$
and give the following definition.

\begin{definition} The singular surface $(M,\ba)$ is:
$$\begin{array}{ll}
\hbox{subcritical}&\hbox{if }\chi(M,\ba)<\t(M,\ba)\\
\hbox{critical}&\hbox{if }\chi(M,\ba)=\t(M,\ba)\\
\hbox{supercritical}&\hbox{if }\chi(M,\ba)>\t(M,\ba).
\end{array}$$
\end{definition}

Observe that in the supercritical case we always have $\chi(M,\ba)>0$. The existence of conformal metrics with prescribed Gaussian curvature on surfaces with conical singularities and geodesic boundary (possibly with corners) in the subcritical regime has been settled down by Troyanov \cite{troy}. Indeed, in this case one is reduced to a minimization problem of a coercive functional. On the contrary, we are not aware of any result concerning the critical/supercritical case. Our main contribution is the following general existence result in the supercritical regime. We will refer to
$$
\G_{\ba}=\biggr\{ 4\pi n +8\pi\sum_{j\in J}(1+\a_j) \,:\, n\in\N\cup\{0\}, \, J\subseteq \{1,\dots,N\} \biggr\}
$$
as the set of critical values. Then, the following holds.

\begin{thm} \label{thm}
Let $(M,\ba)$ be a supercritical singular surface with $\a_j\geq-\frac12$ for $j=1,\dots,N$ and with at least two boundary components. Let $K$ be a positive Lipschitz function on $M$. If $4\pi\chi(M,\ba)\notin\G_{\ba}$, then there exists a conformal metric with Gaussian curvature $K$ on $(M,\ba)$ and geodesic boundary.
\end{thm}

\smallskip

\begin{rem}
It is not difficult to see from the proof that we can allow to have geodesic boundary with corners. More precisely, if $\p M=B_1\sqcup\dots\sqcup B_m$ we can treat the case where the boundary components $B_l$, $l>1$, have corners at the points $q_j\in B_l$ of angles $\theta_{\b_j}=2\pi\left(\frac12+\b_j\right)$ with $\b_j>0$, see Lemma~\ref{lem:ret}.
\end{rem}

\smallskip

\begin{rem}
We stress that we allow to have conical singularities with both positive and negative orders, that is cone angles both less and greater than $2\pi$, provided that $\a_j\geq-\frac12$, see the discussion below.
\end{rem}

\smallskip

The argument is based on Morse theory in the spirit of \cite{bdmm}, where the closed surface (empty boundary) case is considered, by studying the Liouville PDE \eqref{liouv} and its associated functional $J_\l$ given in \eqref{funct}. More precisely, the desired conformal metric will be realised as a min-max solution of \eqref{liouv}, which in turn is produced by the topological changes in the structure of sublevels of $J_\l$. Indeed, high sublevels have trivial topology while we will show that low sublevels are non-contractible.

By means of improved Moser-Trudinger inequalities, the low sublevels can be described by some formal barycenters of $M$, that is family of unit measures supported in a finite number of points of $\ov{M}$. Compared to the classical Liouville equation, the difficulties are due both to the presence of conical singularities and the boundary $\p M$. Indeed, the unit measures may be supported around a conical singularity or on the boundary, which makes the analysis highly non-trivial. We tackle this problem with the following rough idea: we define the \emph{weight} of a point $p\in \ov M$ according to a local Moser-Trudinger inequality around that point, which gives a local volume control in terms of the Dirichlet energy, see Lemma \ref{lem:mt}. We get an \emph{amount} of $8\pi$ near regular points, $8\pi\left(1+\min\{\a_j,0\}\right)$ around a singular point and $4\pi$ for points lying on the boundary $\p M$ (this is natural since the volume around a point on the boundary is half the volume of a point in the interior of the surface). The idea is then to retract the barycenters of $M$ onto those of a boundary component of $M$, where the points have the smallest weight. Here we use the assumptions $\a_j\geq-\frac12$ and the fact that the surface has at least two boundary components. Indeed, we can prove that the barycenters of a boundary component embed non-trivially into arbitrarily low sublevels of $J_\l$ yielding non-trivial topology of the latter.

We stress once more that this idea allows us to consider conical singularities with both positive and negative orders with a simple unified approach. Up to now, positive and negative singularities have been studied separately with different arguments, see for example \cite{bdmm,bm,mal-ru} for the positive case and \cite{car-mal} for the negative one. We mention here that arbitrarily signed singularities have been considered in the sign-changing prescribed curvature problem in \cite{DM-LS} (see also \cite{DM-LS-R}) and for Toda systems in \cite{bat2} with different methods.

\medskip

The paper is organized as follows. In section \ref{sec:liouv} we introduce the Liouville PDE and in section \ref{sec:proof} we prove the main result.

\

\section{The Liouville equation} \label{sec:liouv}

\medskip

In this section we set up the PDE approach. Let $g_0$ be a smooth metric representing any given conformal structure on $M$ and consider a new metric $g=e^u g_0$ with conical singularities at the points $p_j$. The curvatures then transform according to the following law:
\be \label{liouv}
\left\{\begin{array}{ll}
-\D u +2 K_0 = 2 K e^u  & \text{ in } M\setminus\{p_1,\dots,p_N\}, \vspace{0.2cm}\\
\dfrac{\p u}{\p \nu} +2 h_0 = 2he^\frac{u}2  &\text{ on } \p M,\end{array}\right.
\ee
where $\D=\D_{g_0}$ stands for the Laplace-Beltrami operator associated to the metric $g_0$ and $\nu$ is the outward normal vector to $\p M$. Here $K_0$, $K$ are the Gaussian curvatures and $h_0$, $h$ are the  geodesic curvature of the boundary with respect to metrics $g_0$ and $g$, respectively. The conical singularities are encoded by the behavior
$$
u(z)= 2\a_j\log|z|+O(1) \quad \mbox{locally around } p_j.
$$
First, we desingularize the behavior of $u$ around the conical points by writing
$$
v= u+4\pi\sum_{j=1}^N\a_j G_{p_j},
$$
where $G_{p_j}$ is the fundamental solution of the Laplace equation
on $M$ with pole at $p_j$, i.e. the unique solution to
$$
\left\{\begin{array}{ll}
-\D G_{p_j}=\d_{p_j}-\dfrac{1}{|M|}  & \text{ in } M, \vspace{0.2cm}\\
\dfrac{\p G_{p_j}}{\p \nu} = 0  &\text{ on } \p M,\end{array}\right.
$$
with $\int_M G_{p_j}=0$, where $\d_{p_j}$ is the Dirac delta with pole $p_j$ and $|M|$ is the area of $M$. Here and in the rest of the paper the volume form is induced by the metric $g_0$. Next, we can always assume that our initial metric has constant Gaussian curvature and that $\p M$ is geodesic, see for example Proposition 3.1 in \cite{lsmr}. Since we are interested in target metrics inducing geodesic boundary, we are led to consider the following problem:
\be \label{liouv2}
\left\{\begin{array}{ll}
-\D v +2 K_0 +\dfrac{4\pi}{|M|}\displaystyle\sum_{j=1}^N\a_j=  K_{\ba}e^v  & \text{ in } M, \\
\dfrac{\p v}{\p \nu} =0  &\text{ on } \p M,\end{array}\right.
\ee
where $K_0\equiv const.$ and $K_{\ba}=2K e^{-4\pi\sum_{j=1}^N\a_j G_{p_j}}$. Observe that we have the singular behavior
$$
K_{\ba}(z)=C|z|^{2\a_j}(1+O(1)) \quad \mbox{locally around } p_j.
$$
Now, by the Gauss-Bonnet formula (recall that we have flat geodesic) one necessarily has
$$\int_M K_{\ba}e^v = \int_M 2 K_0 + 4\pi\sum_{j=1}^N\a_j=4\pi\chi(M) + 4\pi\sum_{j=1}^N\a_j = 4\pi\chi(M,\ba).$$
We can thus rewrite the problem \eqref{liouv2} as
\be \label{liouv3}
\left\{\begin{array}{ll}
-\D v +2 K_0 +\dfrac{4\pi}{|M|}\displaystyle\sum_{j=1}^N\a_j= \l\dfrac{K_{\ba}e^v}{\int_M K_{\ba}e^v}  & \text{ in } M, \\
\dfrac{\p v}{\p \nu} =0  &\text{ on } \p M,\end{array}\right.
\ee
with $\l=4\pi\chi(M,\ba)$. Therefore, the main existence result in Theorem \ref{thm} will follow once we prove the following result.
\begin{thm} \label{thm2}
Let $M$ be a surface with at least two boundary components. Let $\a_j\geq-\frac12$ for $j=1,\dots,N$ and let $K$ be a positive Lipschitz function on $M$. If $\l\notin\G_{\ba}$, then there exists a solution to \eqref{liouv3}.
\end{thm}

\medskip

The proof will be presented in the next section and is based on the variational structure of the problem. Since the equations in \eqref{liouv3} are
invariant up to an additive constant, we shall restrict ourselves to the subspace of functions with zero average
$$\ov{H}^1(M)=\left\{v\in H^1(M) \,:\, \int_M v=0\right\}$$
and look for critical points of the Euler-Lagrange functional
\be \label{funct}
J_\l(v)=\frac12\int_M |\n v|^2 -\l\log\int_M K_{\ba}e^v, \quad v\in\ov{H}^1(M).
\ee

\

\section{The proof of the main result} \label{sec:proof}

\medskip

By the discussion in the previous section, the existence of a conformal metric as described in Theorem \ref{thm} will follow by solving \eqref{liouv3}. Therefore, we prove here Theorem \ref{thm2} by looking at the functional $J_\l$ given in \eqref{funct}. We are inspired here by the argument proposed in \cite{bdmm}, where the closed surface case is considered. The aim will be to detect
a change of topology between its sublevels
$$
J_\l^a=\left\{v\in\ov{H}^1(M) \,:\, J_\l(v)\leq a\right\}.
$$
We start with the following general blow-up picture, referring to \cite{bt} for what concerns the case of positive singularities, to \cite{bar-mon} for negative singularities and to \cite{bat,ww} for boundary blow-up.
\begin{pro} \label{pro:comp} \emph{(\cite{bar-mon,bt,bat,ww})}
Let $(v_n)_n$ be a sequence of solutions to \eqref{liouv3} with $\l_n\to\l$. Then, up to a subsequence, one of the following alternatives holds:
\begin{itemize}

\item[1.] \emph{(Compactness):} $v_n$ are uniformly bounded.

\medskip

\item[2.] \emph{(Blow-up):} $\max_{\ov M} v_n\to+\infty$ and there exists a finite blow-up set $S=\{q_1,\dots,q_m\}$ such that
$$
\l_n\dfrac{K_{\ba}e^{v_n}}{\int_M K_{\ba}e^{v_n}} \rightharpoonup \sum_{j=1}^m \s_j \d_{q_j}
$$
in the sense of measures, where
$$
\left\{\begin{array}{ll}
 \s_j=8\pi(1+\a_j) & \text{ if } q_j=p_j, \\
 \s_j=8\pi & \text{ if } q_j\in M\setminus\{p_1,\dots,p_N\}, \\
 \s_j=4\pi & \text{ if } q_j\in\p M.\end{array}\right.
$$
\end{itemize}

\medskip

In particular, if $\l\notin\G_{\ba}$ then $v_n$ are uniformly bounded.
\end{pro}

\medskip

The latter compactness property is needed to bypass the Palais-Smale condition, as it was shown in \cite{lucia}, where a deformation lemma is used to derive the following crucial result.
\begin{lem} \label{lem:def}
Suppose $\l\notin\G_{\ba}$ and that $J_\l$ has no critical levels inside $[a,b]$. Then, $J_\l^a$ is a deformation retract of $J_\l^b$.
\end{lem}

\medskip

Next, by Proposition \ref{pro:comp}, $J_\l$ has no critical points above some high level \mbox{$b\gg0$}. Therefore, the deformation Lemma \ref{lem:def} can be applied to obtain the following topological property (see also Corollary 2.8 in \cite{mal}).
\begin{pro} \label{pro:contr}
Suppose $\l\notin\G_{\ba}$. Then, there exists $b \gg 0$ such that $J_\l^b$ is a deformation retract of $\ov H^1(M)$ and it is thus contractible.
\end{pro}

\medskip

We are left with showing that the low sublevels of $J_\l$ are non-contractible. To this end we will need improved versions of the Moser-Trudinger (Troyanov) inequality. This is done by means of a localized version of the Moser-Trudinger inequality, which is based on cut-off functions and spectral decomposition, see for example the approach in Proposition 2.3 in \cite{mal-ru2}, where the Toda system is considered. This idea has its origin in \cite{chen-li1} and has been then adapted by many authors, see for example \cite{car-mal} and \cite{dj} for the singular and regular case, respectively. Observe that the inequality depends on whether we are localizing it around a regular point, a conical point or at the boundary.
\begin{lem} \emph{(\cite{car-mal,chen-li1,dj})} \label{lem:mt} Let $\d>0$ and let $\O\subset \wtilde\O \subset \ov M$ be such that $d\left(\O,\p\wtilde\O\right)>\d$.

\begin{itemize}

\item[1.] \emph{(Regular case):} if $d\left(\wtilde\O,p_j\right)>\d$ for all $j$'s and $d\left(\wtilde\O,\p M\right)>\d$, then there exists $C_{\e,\d}>0$ such that for all $v\in \ov H^1(M)$
$$
8\pi\log \int_{\O} K_{\ba}e^v \leq \frac{1}{2} \int_{\wtilde\O} |\n v|^2+\e\int_M |\n v|^2+ C.
$$

\medskip

\item[2.] \emph{(Singular case):} if $p_j\in\O$ for some $j$, $d\left(\wtilde\O,p_l\right)>\d$ for all $l\neq j$ and $d(\O,\p M)>\d$, then there exists $C_{\e,\d}>0$ such that for all $v\in \ov H^1(M)$
$$
8\pi\left(1+\min\{\a_j,0\}\right)\log \int_{\O} K_{\ba}e^v \leq \frac{1}{2} \int_{\wtilde\O} |\n v|^2+\e\int_M |\n v|^2+ C.
$$

\medskip

\item[3.] \emph{(Boundary case):} if $d\left(\wtilde\O,p_j\right)>\d$ for all $j$'s, then there exists $C_{\e,\d}>0$ such that for all $v\in \ov H^1(M)$
$$
4\pi\log \int_{\O} K_{\ba}e^v \leq \frac{1}{2} \int_{\wtilde\O} |\n v|^2+\e\int_M |\n v|^2+ C.
$$
\end{itemize}
\end{lem}

\medskip

The above different scenarios make the Morse approach for Liouville equations in supercritical regimes quite challenging, especially in the case where both positive and negative singularities are present. Here we avoid such complexity by using the following idea: we define the \emph{weight} of a point $p\in \ov M$ according to the constant in the above local Moser-Trudinger inequalities around that point, which indicates the local volume control in terms of the Dirichlet energy. Therefore, regular points have weight $8\pi$, singular points $8\pi\left(1+\min\{\a_j,0\}\right)$ and boundary points $4\pi$. The idea is to focus on points with the smallest weight, that is points on the boundary (recall $\a_j\geq-\frac12$), through a suitable retraction, see the discussion later on. Roughly speaking, we will prove that the boundary generates non-trivial homology of the low sublevels of $J_\l$. This is done by describing the functions in the low sublevels by means of configurations supported on the boundary and by showing that the non-trivial homology groups of the boundary inject into the homology groups of the low sublevels, see \eqref{inj} for more details.

\medskip

To this end, we start by observing that, in Lemma \ref{lem:mt}, a concentration of the conformal volume $K_{\ba}e^v$ in $\O$, in the sense
$$
\int_{\O} \frac{K_{\ba}e^v }{\int_M K_{\ba}e^v} \geq \g, \quad \mbox{for some } \g>0,
$$
would give a global volume control in terms of the Dirichlet energy. We conclude that whenever $K_{\ba}e^v$ is concentrated in different regions of $\ov M$ an improved Moser-Trudinger inequality holds just by summing up the local inequalities. Improved inequalities in turn give improved lower bounds on the functional $J_\l$. Therefore, in the low sublevels, $K_{\ba}e^v$ cannot be concentrated in too many different regions, i.e. we have the following property. Here we have just to observe that, since $\a_j\geq-\frac12$, any local Moser-Trudinger inequality in Lemma \ref{lem:mt} gives a volume control of at least $4\pi$.
\begin{lem}
Suppose $\l < 4(k+1)\pi$. Then, for any $\e,r>0$, there exists $L=L(\e,r)>0$ such that for any $v\in J_\l^{-L}$ there exist $k$ points $\{q_1,\dots,q_k\}\subset \ov M$ such that
\be\label{conc}
\int_{\bigcup_{i=1}^k B_r(q_i)} \frac{K_{\ba}e^v }{\int_M K_{\ba}e^v} \geq 1-\e.
\ee
\end{lem}

\begin{proof}
Assume \eqref{conc} does not hold. Then, by a standard covering lemma (see for instance \cite{mal}, Lemma 3.3), there exists $\d>0$ and $\O_1,\dots,\O_{k+1}\subset \ov M$ such that
$$d(\O_i,\O_j)\ge2\d,\;\forall i\ne j\quad\quad\quad\int_{\O_j} \frac{K_{\ba}e^v }{\int_M K_{\ba}e^v} \ge\d.$$
For any $j=1,\dots,k+1$ we apply Lemma \ref{lem:mt} with $\O=\O_j,\wtilde\O=B_\d(\O_j)$. In the \emph{boundary case}, we get
\begin{eqnarray*}
4\pi\log\int_M K_{\ba}e^v&\leq&4\pi\log\int_{\O_j} K_{\ba}e^v+4\pi\log\frac1\delta\\
&\leq&\frac{1}{2} \int_{B_\d(\O_j)} |\n v|^2+\e\int_M |\n v|^2+ C.
\end{eqnarray*}
In the \emph{regular case}, since Jensen's inequality gives $\log\int_M K_{\ba}e^v\ge-C$, then
\begin{eqnarray*}
4\pi\log\int_M K_{\ba}e^v&\leq&8\pi\log\int_M K_{\ba}e^v+C\\
&\leq&8\pi\log\int_{\O_j} K_{\ba}e^v+C\\
&\leq&\frac{1}{2} \int_{B_\d(\O_j)} |\n v|^2+\e\int_M |\n v|^2+ C.
\end{eqnarray*}
The same computation holds true in the \emph{singular case}, since $8\pi\left(1+\min\{\a_j,0\}\right)\ge4\pi$; therefore, summing on all $j$'s and taking account that $B_\d(\O_i)\cap B_\d(\O_j)=\emptyset$ for $i\ne j$, we get
$$4(k+1)\pi\log\int_M K_{\ba}e^v\leq\left(\frac{1}{2}+k\e\right)\int_M |\n v|^2+C.$$
It follows that
$$
	J_\l(v)\geq \frac12\left(1-\frac{\l}{4(k+1)\pi}(1+2k\e) \right)\int_M |\n v|^2 -C.
$$
Since $\l<4(k+1)\pi$, we can choose $\e>0$ such that $4(k+1)\pi=\l(1+2k\e)$ and we get $J_\l(v)\ge -L$ for some $L>0$, which concludes the proof.
\end{proof}

\medskip

This naturally leads us to describe the low sublevels by unit measures supported in (at most) $k$ points of $\ov M$, known as the formal barycenters of $M$ of order $k$:
\begin{equation}\label{sigk}
M_k = \left\{ \sum_{i=1}^k t_i\delta_{q_i} \, : \, \sum_{i=1}^k t_i=1,t_i\geq 0,q_i\in \ov M,\forall\,i=1,\dots,k \right\}.
\end{equation}
Indeed, one can project the measure $\frac{K_{\ba}e^v }{\int_M K_{\ba}e^v}$ on the closest element in $M_k$, see Lemma 4.9 in \cite{dj}.
\begin{pro} \label{pro:proj}
Suppose $\l < 4(k+1)\pi$. Then, there exists a projection $\Psi: J_\l^{-L}\to M_k$, for some $L\gg0$.
\end{pro}

\medskip

Now, to restrict our target on barycenters supported only on the boundary, we need the following result. We recall that $M$ is assumed to have at least two boundary components and we write 
$$
\p M=B_1\sqcup\dots\sqcup B_m,
$$ 
with $m>1$, where $B_i\simeq\S^1$.
\begin{lem} \label{lem:ret}
Suppose $\l < 4(k+1)\pi$. Then, there exists a map $\Psi_{\Pi}: J_\l^{-L}\to (B_1)_k$, for some $L\gg0$.
\end{lem}

\begin{proof}
We start by defining a global retraction $\Pi:M\to B_1$. To this end, consider the space $\R^3\ni(x,y,z)$ and the projection $P:\R^3\to\{z=0\}$. We point out that any two compact surfaces with the same genus and same number of boundary components are homeomorphic. Therefore, we can assume without loss of generality that $M$ is embedded in $\R^3$ such that in the holes $B_1$ and $B_2$ passes the same line parallel to the $z$-axis such that $P(M)$ is a disk with at least one hole $H$ with $\p H=P(B_1)$, see Figure~\ref{fig:retraction}. Therefore, there exists a retraction $R:P(M)\to P(B_1)$ which induces a retraction $\Pi:M\to B_1$.

\medskip

\begin{figure}[h]
\centering
\includegraphics[width=0.7\linewidth]{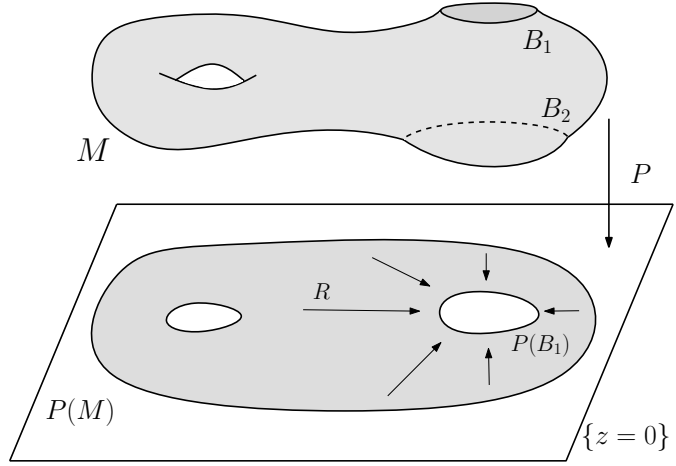}
\caption{The construction of the retraction $\Pi:M\to B_1$.}
\label{fig:retraction}
\end{figure}

\medskip

Now, by Proposition \ref{pro:proj} there exists a projection $\Psi: J_\l^{-L}\to M_k$, for some \mbox{$L\gg0$.} The desired map $\Psi_{\Pi}: J_\l^{-L}\to (B_1)_k$ is then defined through the composition
$$
J_\l^{-L} \stackrel{\Psi}{\longrightarrow} M_k \stackrel{\Pi_*}{\longrightarrow} (B_1)_k,
$$
where $\Pi_*$ denotes here the push-forward of measures induced by the above retraction.
\end{proof}

\medskip

To gain some topological properties of the low sublevels we construct now a reverse map $\Phi: (B_1)_k\to J_\l^{-L}$. To this end, we consider a family of regular bubbles centered at the boundary component $B_1$ and we set, for $\L>0$, $\Phi: (B_1)_k\to \ov{H}^1(M)$ as
\be \label{phi}
\Phi:\s=\sum_{i=1}^k t_i\d_{q_i}\quad\mapsto\quad\varphi_{\L,\s} - \ov{\varphi}_{\L,\s},
\ee
where
$$
\varphi_{\L,\s}(y) = \log \sum_{i=1}^k t_i \left( \frac{\L}{1+\L^2 d(y,q_i)^2} \right)^2,$$
and $\ov{\varphi}_{\L,\s}$ is the average of $\varphi_{\L,\s}$. Then, the following estimates hold true.
\begin{lem} \label{lem:test}
Let $\varphi_{\L,\s}$, $\s\in (B_1)_k$ be the functions defined above. Then, for \mbox{$\L\to+\infty$} we have
\begin{eqnarray*}
\frac12 \int_M |\n \varphi_{\L,\s}|^2  &\leq& 8k\pi(1+o(1))\log\L, \label{dir}\\
\log \int_M K_{\ba}e^{\varphi_{\L,\s}-\ov{\varphi}_{\L,\s}}  &=& 2(1+o(1))\log\L. \nonumber
\end{eqnarray*}
Moreover,
\be \label{conv}
\frac{K_{\ba}e^{\varphi_{\L,\s}}}{\int_M K_{\ba}e^{\varphi_{\L,\s}}} \rightharpoonup \s\in (B_1)_k,
\ee
in the sense of measures.
\end{lem}

\medskip

The latter estimates are by now standard and we refer for example to Proposition~4.2 in \cite{mal2} where the regular closed surface case is considered. The presence of the boundary can be handled with obvious modifications. Indeed, the only difference is that the main contribution of the Dirichlet energy \eqref{dir} comes from half-balls around the centers of the bubbles and it is thus divided by a factor $2$. Observe that we can neglect the effect of the singularities since we are considering bubbles centered on the boundary component $B_1$ which does not have conical points.

By plugging the above estimates into the functional $J_\l$ it is then easy to conclude the following.
\begin{pro} \label{pro:test}
Suppose $\l>4k\pi$ and let $\Phi$ be given as in \eqref{phi}. Then, for any $L>0$ there exists $\L\gg0$ such that $\Phi: (B_1)_k\to J_\l^{-L}$.
\end{pro}
\begin{proof}
Indeed, we have
$$
	J_\l(\Phi(\s))\leq (8k\pi-2\l+o(1))\log\L
$$
and since $\l>4k\pi$ the thesis follows by choosing suitably $\L\gg0$.
\end{proof}

\medskip

We can now prove the main result.

\medskip

\begin{proof}[Proof of Theorem \ref{thm2}.]
Take $\l\in(4k\pi,4(k+1)\pi)\setminus\G_{\ba}$. By Lemma \ref{lem:ret} there exists a map $\Psi_{\Pi}: J_\l^{-L}\to (B_1)_k$, for some $L\gg0$. On the other hand, by Proposition~\ref{pro:test} we have a map $\Phi: (B_1)_k\to J_\l^{-L}$. Such maps are natural in the sense that the composition
$$\begin{array}{ccccc}
(B_1)_k&\stackrel{\Phi}{\longrightarrow}&J_\l^{-L}&\stackrel{\Psi_{\Pi}}{\longrightarrow}& (B_1)_k \\
\s&\mapsto&\left(\varphi_{\L,\s}-\ov{\varphi}_{\L,\s}\right)&\mapsto&\frac{K_{\ba}e^{\varphi_{\L,\s}}}{\int_M K_{\ba}e^{\varphi_{\L,\s}}} \simeq \s
\end{array}$$
is homotopic to the identity on $(B_1)_k$. Here we recall \eqref{conv} holds true. We refer to Proposition 4.4 in \cite{mal2} for more details on this point. Passing to the induced maps $\Phi^*,\Psi_{\Pi}^*$ between homological groups $H_*$ we derive $\Psi_{\Pi}^*\circ\Phi^*=\mbox{Id}_{(B_1)_k}^*$. In particular,
\begin{equation} \label{inj}
H_*((B_1)_k) \hookrightarrow H_*\left(J_\l^{-L}\right)
\end{equation}
injectively. Since $(B_1)_k\simeq\left(\S^1\right)_k\simeq\S^{2k-1}$, see for example Proposition 3.2 in \cite{bm}, we deduce that $J_\l^{-L}$ is not contractible. But $J_\l^{b}$ is contractible for $b\gg0$ by Proposition~\ref{pro:contr}. The existence of a solution to \eqref{liouv3} follows by Lemma~\ref{lem:def}.
\end{proof}

\

\end{document}